\documentclass[oneside]{amsart}
\usepackage[colorlinks=true,allcolors=blue!80!black]{hyperref}
\usepackage{tikz, tikz-cd}
\usetikzlibrary{matrix, snakes, patterns, shadows.blur, backgrounds}

\usepackage{amssymb, amsthm}
\usepackage{enumerate, amsfonts, amsxtra,amsmath,latexsym,epsfig, color}
\usepackage{mathrsfs, MnSymbol}
\usepackage{geometry}                		% See geometry.pdf to learn the layout options. There are lots.
\usepackage{graphicx}
\usepackage{subcaption}
\usepackage{autobreak}
\usepackage{color}
\usepackage{amsmath,amscd}	
\usepackage{mathtools}
\usepackage{stmaryrd}
\usepackage{animate}
\usepackage{tikz-cd}
\usepackage{soul}
\usepackage{faktor}
\usepackage{hyperref}

\usepackage{xypic}

\input xy
\xyoption{all}

\newtheorem {theorem}{Theorem}[section]
\newtheorem {lemma} [theorem] {Lemma}
\newtheorem {proposition} [theorem] {Proposition}
\newtheorem {corollary} [theorem] {Corollary}

\theoremstyle{definition}
\newtheorem{definition}[theorem]{Definition}
\newtheorem{remark}[theorem]{Remark}

\newcommand{\leftQ}[2]{\left.\raisebox{-.2em}{$#2$}\middle\backslash\raisebox{.2em}{$#1$}\right.}

\def\co{\colon\thinspace}

\def\C {\mathbb C}
\def\R {\mathbb R}

\def\Hy{\mathbf{H}}

\def\D{\mathbb{D}}

\def\mc {\mathcal}

\def\Z {\mathbb{Z}}

\def\cP{\mathcal{P}}

\DeclareMathOperator{\PU}{PU}
\DeclareMathOperator{\Isom}{Isom}
\DeclareMathOperator{\SU}{SU}
\DeclareMathOperator{\SO}{SO}
\DeclareMathOperator{\Sp}{Sp}
\DeclareMathOperator{\PSL}{PSL}

\begin{document}
\title[Relatively geometric actions of K\"ahler groups]{Relatively geometric actions of K\"ahler groups on $\mathrm{CAT}(0)$ cube complexes}
%\author{Corey Bregman and Kejia Zhu}
%\address{Department of Mathematics, Statistics and Computer Science, University of Illinois at Chicago}

\author{Corey Bregman}
\address{Department of Mathematics and Statistics\\ University of Southern Maine}
\email{corey.bregman@maine.edu}
\urladdr{https://sites.google.com/view/cbregman}

\author{Daniel Groves}
\address{Department of Mathematics, Statistics and Computer Science\\ University of Illinois at Chicago}
\email{dgroves@uic.edu}
\urladdr{http://homepages.math.uic.edu/$\sim$groves/} 

\author{Kejia Zhu}
\address{Department of Mathematics, Statistics and Computer Science\\ University of Illinois at Chicago}
\email{kzhu14@uic.edu}
\urladdr{https://sites.google.com/view/kejiazhu}

\date{}
\maketitle
%\tableofcontents

\begin{abstract}
We prove that for $n\geq 2$, a non-uniform lattice in $\text{PU}(n,1)$ does not admit a relatively geometric action on a $\mathrm{CAT}(0)$ cube complex, in the sense of \cite{einstein2020relative}. As a consequence, if $\Gamma$ is a non-uniform lattice in a non-compact semisimple Lie group $G$ without compact factors that admits a relatively geometric action on a $\mathrm{CAT}(0)$ cube complex, then $G$ is commensurable with $\SO(n,1)$. We also prove that if a K\"ahler group is hyperbolic relative to residually finite parabolic subgroups, and acts relatively geometrically on a $\mathrm{CAT}(0)$ cube complex, then it is virtually a surface group.
%We also prove that given a relatively hyperbolic group with residually finite parabolic subgroups, if it is K\"ahler and acts relatively geometrically on a $\mathrm{CAT}(0)$ cube complex, then it is virtually a surface group.
\end{abstract}

\section{Introduction}

A finitely generated group is called \emph{cubulated} if it acts properly cocompactly on a $\mathrm{CAT}(0)$ cube complex.  Agol \cite{AgolHaken}, building on the work of Wise \cite{wise2012structure} and many others, proved that cubulated hyperbolic groups enjoy many important properties, and used this to solve several open conjectures in 3-manifold topology, in particular the Virtual Haken and Virtual Fibering Conjectures. Wise \cite[$\S17$]{wise2012structure} proved the Virtual Fibering Conjecture in the non-compact, finite-volume setting, using the relatively hyperbolic structure of the fundamental group.

Einstein--Groves define the notion of a \emph{relatively geometric} action of a group pair $(\Gamma,\cP)$ on a $\mathrm{CAT}(0)$ cube complex \cite{einstein2020relative}. For such an action, elements of $\cP$ act elliptically.  This allows the possibility that even though the elements of $\cP$ might not act properly on any $\mathrm{CAT}(0)$ cube complex, there still may be a relatively geometric action. 
 Relatively geometric actions are a natural generalization of proper actions and share many of the same features as in the proper case, especially when $\Gamma$ is hyperbolic relative to $\cP$ .  

 Uniform lattices in $\SO(3,1)$ always act geometrically thus relatively geometrically on $\mathrm{CAT}(0)$ cube complexes  \cite{bergeron2012boundary}. Bergeron--Haglund--Wise \cite{BHW} prove that in higher dimensions, lattices in $\SO(n,1)$ which are arithmetic of simplest type are cubulated.  It also follows from this and Wise's quasi-convex hierarchy theorem \cite{wise2012structure} that many ``hybrid" hyperbolic $n$-manifolds have cubulated fundamental groups. 
 In the relatively geometric setting, using the work of Cooper--Futer \cite{CooperFuter}, Einstein--Groves proved that non-uniform lattices in $\SO(3,1)$ also admit relatively geometric actions, relative to their cusp subgroups \cite{einstein2020relative}. In fact, they prove that if $(G,\cP)$ is hyperbolic relative to free abelian subgroups and the Bowditch boundary $\partial (G,\cP)$ is homeomorphic to $S^2$, then $G$ is isomorphic to a non-uniform lattice in $\SO(3,1) $ if and only if $(G,\cP)$ admits a relatively geometric action on a $\mathrm{CAT}(0)$ cube complex. This result is a relative version of the work of Markovic \cite{Markovic2013} and Ha\"issinsky \cite{Haissinsky2015} in the convex-cocompact setting, giving an equivalent formulation of the Cannon conjecture in terms of actions on hyperbolic $\mathrm{CAT}(0)$ cube complexes. 
It is not known in general whether the above results extend to all lattices in $\SO(n,1)$ for $n\geq 3$. 

In contrast, work of Delzant--Gromov implies that uniform lattices in $\PU(n,1)$ are not cubulated \cite{delzant2005cuts}. 
Recall that a group $\Gamma$ is \emph{K\"ahler} if $\Gamma\cong \pi_1(X)$ for some compact K\"ahler manifold $X$. If $\Gamma\leq \PU(n,1)$ is a torsion-free, uniform lattice, then $\Gamma$ acts freely, properly discontinuously cocompactly on complex hyperbolic $n$--space $\Hy^n_\C$.  The quotient $M=\leftQ{\Hy_\C^n}{\Gamma}$ is a closed, negatively curved K\"ahler manifold, and in particular $\Gamma$ is a hyperbolic, K\"ahler group. In this context, Delzant--Gromov showed that any infinite K\"ahler group that is hyperbolic and cubulated is commensurable to a surface group of genus $g\geq 2$ \cite{delzant2005cuts}. Thus $\Gamma$ is not cubulated for $n\geq 2$.  Since every uniform lattice in $\PU(n,1)$ is virtually torsion-free,  it follows that uniform lattices in $\PU(n,1)$ are not cubulated if $n\geq 2$. 
 
On the other hand, uniform lattices in $\PU(1,1)=\SO(2,1)$, are finite extensions of hyperbolic surface groups, hence are hyperbolic and cubulated. Similarly, non-uniform lattices in $\PU(1,1)$ are the orbifold fundamental groups of surfaces with finitely many cusps, hence virtually free. Such lattices admit both proper cocompact and  relatively geometric actions on $\mathrm{CAT}(0)$ cube complexes.  Since the cusp subgroups of a non-uniform lattice in $\PU(n,1)$ ($n \ge 2)$ are virtually nilpotent but not virtually abelian, it follows from a result of Haglund \cite{haglund2021isometries} that such a lattice does not admit a proper action on a $\mathrm{CAT}(0)$ cube complex (see Proposition~\ref{prop:No Proper Action} below).

However, the parabolic subgroups do not yield such an obstruction to the existence of a relativeley geometric action.  Thus, this leaves open the question of whether non-uniform lattices in $\PU(n,1)$ admit relatively geometric actions on $\mathrm{CAT}(0)$ cube complexes for $n\geq 2$.  Our first result answers this question in the negative.

\begin{theorem}\label{thm:NoRelativeAction} Let $\Gamma\leq \PU(n,1)$ be a non-uniform lattice with $n\geq 2$, and let $\cP$ be the collection of cusp subgroups of $\Gamma$.  Then $(\Gamma,\cP)$ does not admit a relatively geometric action on a $\mathrm{CAT}(0)$ cube complex.  
\end{theorem}
 
\begin{corollary}\label{cor:SSClassification}Let $\Gamma$ be a lattice in a non-compact semisimple Lie group $G$ without compact factors.  If either 
\begin{enumerate}
\item $\Gamma$ is uniform and cubulated hyperbolic, or
\item $\Gamma$ is non-uniform, hyperbolic relative to its cusp subgroups $\mathcal{P}$, and $(\Gamma,\cP)$ admits a relatively geometric action on a $\mathrm{CAT}(0)$ cube complex,
\end{enumerate}
then $G$ is commensurable to $\SO(n,1)$ for some $n\geq 1$.
\end{corollary}

\begin{proof}  A uniform lattice (resp. non-uniform lattice) $\Gamma$ in a semisimple Lie group $G$ is hyperbolic (resp. hyperbolic relative to its cusp subgroups $\cP$) if and only if $G$ has rank 1, by a result of Behrstock--Dru\c{t}u--Mosher \cite{behrstock2009thick}. Any rank 1 noncompact semisimple Lie group is commensurable with one of $\SO(n,1)$, $\PU(n,1)$, $\Sp(n,1)$ for $n\geq 2$, or the isometry group of the octonionic hyperbolic plane $\Hy^2_\mathbb{O}$.  The latter and $\Sp(n,1)$ have Property (T), while $\SO(n,1)$ and  $\PU(n,1)$ do not. Hence if $\Gamma$ is commensurable with a lattice in $ \Sp(n,1)$ or $\Isom(\Hy^2_\mathbb{O})$, then $\Gamma$ has (T). 

By a result of Niblo--Reeves \cite{NibloReeves}, any action of a group with Property (T) on a $\mathrm{CAT}(0)$ cube complex has a global fixed point, so lattices in $\Sp(n,1)$ and $\Isom(\Hy^2_\mathbb{O})$  admit neither geometric nor relatively geometric actions on $\mathrm{CAT}(0)$ cube complexes. Hence if $\Gamma$ is as in the statement of the result, it must be commensurable to a lattice in either $\PU(n,1)$ or $\SO(n,1)$.  For $n\geq 2$, the uniform case of $\Gamma\leq \PU(n,1)$ is eliminated by work Delzant--Gromov \cite{delzant2005cuts}. The corollary now follows from Theorem \ref{thm:NoRelativeAction}.  
\end{proof}

We say that a relatively hyperbolic group pair $(\Gamma,\cP)$ is \emph{properly} relatively hyperbolic if $\cP \ne \{ \Gamma \}$. The following result considers general relatively geometric actions of K\"ahler relatively hyperbolic groups on $\mathrm{CAT}(0)$ cube complexes (when the peripheral subgroups are residually finite).

\begin{theorem}\label{thm:NoKahlerAction}
	Let $(\Gamma,\mathcal P)$ be a properly relatively hyperbolic pair such that each element of $\mathcal{P}$ is residually finite. If $\Gamma$ is K\"ahler and acts relatively geometrically on a $\mathrm{CAT}(0)$ cube complex, then $\Gamma$ is virtually a hyperbolic surface group.
\end{theorem}
We will deduce Theorem \ref{thm:NoRelativeAction} from Theorem \ref{thm:NoKahlerAction} in Section \ref{sec:No Lattice Action}. In fact, non-uniform lattices in $\PU(n,1)$ are K\"ahler for $n\geq 3$ \cite{Toledo}, hence Theorem \ref{thm:NoRelativeAction} follows immediately from Theorem \ref{thm:NoKahlerAction} in this range. However, our proof of Theorem \ref{thm:NoRelativeAction} will work for all $n\geq 2$, and will not use this fact. In \cite{DelzantPy}, Delzant--Py considered actions of K\"ahler groups on locally finite, finite-dimensional $\mathrm{CAT}(0)$ cube complexes that are more general than geometric ones (see Theorem A for precise hypotheses), and showed that every such action virtually factors through a surface group. We remark that the cube complexes appearing in relatively geometric actions will in general not be locally finite. 

We conclude the introduction with a sample application of Theorem~\ref{thm:NoKahlerAction}.

\begin{corollary}
    Suppose that $A$ and $B$ are infinite residually finite groups which are not virtually free.  Any $C'\left(\frac{1}{6}\right)$--small cancellation quotient of $A\ast B$ is not K\"ahler.
\end{corollary}
\begin{proof}
    Let $\Gamma$ be such a small cancellation quotient of $A \ast B$.  According to \cite{einstein-ng}, $\Gamma$ is residually finite and admits a relatively geometric action on a $\mathrm{CAT}(0)$ cube complex.  If $\Gamma$ were K\"ahler, it would be a virtual surface group, by Theorem~\ref{thm:NoKahlerAction}.  However, $A$ embeds in $\Gamma$ as an infinite-index subgroup, and the only infinite index subgroups of virtual surface groups are virtually free.
\end{proof}

\noindent
\textbf{Outline:} In Section \ref{sec:ReviewActions}, we review the definition of a relatively geometric action of a group pair on a $\mathrm{CAT}(0)$ cube complex and the notion of group-theoretic Dehn fillings, then collect some known results about these. In Section~\ref{sec:ending} we prove Theorem~\ref{thm:NoKahlerAction}. In Section \ref{sec:No Lattice Action}, after reviewing the Borel--Serre and toroidal compactifications of non-uniform quotients of complex hyperbolic space, we prove Theorem \ref{thm:NoRelativeAction}.\\

\noindent\textbf{Acknowledgments:} The first author was supported by NSF grant DMS-2052801. The second author was supported by NSF grants DMS-1904913 and DMS-2203343.  The third author would like to thank his advisor, Daniel Groves, for introducing him to the subject and answering his questions. He would like to thank his co-advisor, Anatoly Libgober, for his constant support and warm encouragement. He would also like thank Hao Liang and Xuzhi (Carl) Tang for helpful discussions.\\

\section{Actions on $\mathrm{CAT}(0)$ Cube Complexes}\label{sec:ReviewActions}
In this section, we review the notion of a relatively geometric action of a group pair $(\Gamma,\mc{P})$ on a $\mathrm{CAT}(0)$ cube complex, defined by Einstein and Groves in \cite{einstein2020relative}. We then introduce Dehn fillings of group pairs and recall some useful results from \cite{einstein2022relatively}.  

\begin{definition}\label{def:drg}
	Let $\Gamma$ be a group and $\cP$ a collection of subgroups of $\Gamma$. An action of $\Gamma$ on a $\mathrm{CAT}(0)$ cube complex $X$ is \emph{relatively geometric with respect to $\mathcal P$} if
    \begin{enumerate}
        \item $\leftQ{X}{\Gamma}$ is compact;
        \item Each element of $\mathcal P$ acts elliptically on $X$;
        \item Each cell stabilizer in $X$ is either finite or else conjugate to a finite-index subgroup of an element of $\mathcal P$.
    \end{enumerate}
\end{definition}

Recall that if $(\Gamma,\cP)$ is a relatively hyperbolic group pair and $\Gamma_0\leq \Gamma$ has finite-index then $(\Gamma_0,\cP_0)$ is also a relatively hyperbolic group pair, where $\cP_0$ is the set of representatives of the $\Gamma_0$--conjugacy classes of\begin{equation}\label{eqn:FIPeripherals}\left\{P^g\cap\Gamma_0 \mid g\in \Gamma, P\in \mathcal P\right\}\end{equation}

Since $[\Gamma\co \Gamma_0]$ is finite, $\cP_0$ is still a finite collection of subgroups.  It follows that  if $\Gamma$ admits a relatively geometric action on a $\mathrm{CAT}(0)$ cube complex $X$, then $(\Gamma_0,\cP_0)$ also admits a relatively geometric action on $X$ by restriction.  Indeed, (2) and (3) in Definition \ref{def:drg} follow immediately and $(1)$ follows from that fact that under the natural map $\leftQ{X}{\Gamma_0} \rightarrow \leftQ{X}{\Gamma}$, each cell of $\leftQ{X}{\Gamma}$ has at most $[\Gamma\co \Gamma_0]<\infty$ pre-images. Hence if $\leftQ{X}{\Gamma}$ is compact, so is $\leftQ{X}{\Gamma_0}$.  We have just proven 

\begin{lemma}\label{lem:Finite Index Action} Let $\Gamma_0\leq \Gamma$ be a finite-index subgroup.  If $(\Gamma,\cP)$ has a relatively geometric action on a $\mathrm{CAT}(0)$ cube complex $X$, then the restriction of this action to $(\Gamma_0,\cP_0)$ is also relatively geometric, where $\cP_0$ is defined as in Equation \ref{eqn:FIPeripherals}.
\end{lemma}

\subsection{Dehn fillings} 
\quad\\
Dehn fillings first appeared in the context of 3-manifold topology and were subsequently generalized to the group-theoretic setting by Osin \cite{osin2007peripheral} and Groves--Manning \cite{GrovesManningDehn}. We now recall the notion of a Dehn filling of a group pair $(G,\mathcal{P})$:

\begin{definition}[Dehn Filling]
 Given a group pair $(G,\mathcal P)$, where $\mathcal P = \{P_1, ..., P_m\}$ and a choice of normal subgroups of peripheral groups $\mathcal{N}=\{N_i\unlhd P_i\}$,  the \emph{Dehn filling} of $(G,\mathcal P)$ with respect to $\mathcal{N}$ is the pair $(\overline{G},\overline{\mathcal P})$ where $\overline{G}=G/K$ and $K=\llangle \cup N_i\rrangle$ is the normal closure in $G$ of the group generated by $\{\cup_i N_i\}$ and $\bar{\mathcal P}$ is the set of images of $\mathcal P$ under this quotient. The $N_i$ are called the \emph{filling kernels}. When we want to specify the filling kernels we write $G(N_1,\ldots, N_m)$ for the quotient $\overline{G}$.
\end{definition}

\begin{definition}[Peripherally finite]
	 If each normal subgroup $N_i$ has finite-index in $P_i$, the filling is said to be \emph{peripherally finite}. 
\end{definition}

\begin{definition}[Sufficiently long]
	 We say that a property $\mathcal X$ holds for all sufficiently long Dehn fillings of $(G,\mathcal P)$ if there is a finite subset $B\subset G\smallsetminus \{ 1 \}$ so that whenever $N_i \cap B =\emptyset$ for all $i$, the corresponding Dehn filling $G(N_1,..., N_n)$ has property $\mathcal X$.
\end{definition}

The proof of the next theorem relies on the notion of a $\mathcal{Q}$--filling of a collection of subgroups $\mc{Q}$ of $G$.  Recall from \cite{groves2018hyperbolic} that given a subgroup $Q<G$, the quotient $G(N_1,\ldots,N_m)$ is a \emph{$Q$--filling} if for all $g \in G$, and $P_i\in \cP$,  $|Q \cap P_i^g|=\infty$ implies $N_i^g\subseteq Q$. If $\mc{Q}=\{Q_1,\ldots, Q_l\}$ is a family of subgroups, then $G(N_1,\ldots, N_m)$ is a \emph{$\mc{Q}$--filling} if it is a $Q$--filling for every $Q\in \mc{Q}$.

Let $\mathcal{Q}$ be a collection of finite-index subgroups of elements of $\mathcal{P}$ so that any infinite cell stabilizer contains a conjugate of an element of $\mathcal{Q}$. The following is proved in \cite{einstein2022relatively}.

\begin{theorem}[Proposition 4.1 and Corollary 4.2 of \cite{einstein2022relatively}] \label{thm:DehnFilling}Let $(\Gamma,\cP)$ be a relatively hyperbolic pair such that the elements of $\mathcal{P}$ are residually finite.   If $(\Gamma,\cP)$ admits a relatively geometric action on a $\mathrm{CAT}(0)$ cube complex $X$ then 
\begin{enumerate}
    \item\label{item:quotientCCC} For sufficiently long $\mathcal{Q}$--fillings $\Gamma\rightarrow \overline{\Gamma}=\Gamma/K$, the quotient $\overline{X}=\leftQ{X}{K}$ is a $\mathrm{CAT}(0)$ cube complex; and
    \item\label{item:hypQ} Any sufficiently long, peripherally finite $\mathcal{Q}$--filling of $\Gamma$ is hyperbolic and virtually special. 
\end{enumerate}
\end{theorem}

The following result is implicit in \cite{einstein2022relatively}.  For completeness, we provide a proof.
\begin{lemma}  \label{lem:quotientRG}
In the context of Theorem~\ref{thm:DehnFilling}.\eqref{item:quotientCCC}, the action of $\overline{\Gamma}$ on $\overline{X}$ is relatively geometric.
\end{lemma}
\begin{proof}
Since $\leftQ{\overline{X}}{\overline{\Gamma}} = \leftQ{X}{\Gamma}$ the action is cocompact.  Let $\overline{\cP}$ be the induced peripheral structure on $\overline{\Gamma}$ (the image of $\cP$).  The fact that elements of $\overline{\cP}$ act elliptically on $\overline{X}$ follows from the fact that elements of $\mathcal{P}$ act elliptically on $X.$ Because each cell-stabilizer of $\Gamma\curvearrowright X$ is either finite or conjugate to a finite-index of subgroup of some $P_i\in \mathcal{P}$, this implies that the cell-stabilizers of  $\overline{\Gamma} \curvearrowright \overline{X}$ are conjugate to finite-index subgroups of $P_i/(K\cap P_i)$ (the elements of $\overline{\cP}$). Thus the action of $\Delta$ on $Y$ is relatively geometric.
\end{proof}

\section{Relatively geometric actions: the K\"ahler case}\label{sec:ending}

In this section, we apply Theorem \ref{thm:DehnFilling} to prove Theorem \ref{thm:NoKahlerAction}.  The main idea is to use Dehn filling to produce a minimal action of a finite-index subgroup of $\Gamma$ on a tree with finite kernel.  A deep result of Gromov--Schoen implies that any K\"ahler group admitting a minimal acting on tree with finite kernel must be virtually a hyperbolic surface group \cite{gromov1992harmonic}.  

\begin{proof}[Proof of Theorem \ref{thm:NoKahlerAction}]
Suppose that $(\Gamma,\mathcal P)$ acts relatively geometrically on a $\mathrm{CAT}(0)$ cube complex. Since the elements of $\mathcal{P}$ are residually finite, there exists a sufficiently long, peripherally finite $\mathcal{Q}$--filling $ \Gamma \to \overline{\Gamma}= \Gamma/K$ which satisfies the hypotheses of Theorem~\ref{thm:DehnFilling}.(\eqref{item:hypQ}), so $\overline{\Gamma}$ is hyperbolic and $\overline{X}=\leftQ{X}{K}$ is a $\mathrm{CAT}(0)$ cube complex. Let $\Gamma_0\leq \Gamma$ be a finite-index subgroup such that $\Gamma_0$ is torsion-free and $\leftQ{X}{\Gamma_0}$ is special, which exists by \cite[Theorem 1.4]{einstein2022relatively}.

Cutting along an embedded essential two-sided hyperplane $H$ in $\leftQ{X}{\Gamma_0}$ yields a splitting of $\Gamma_0$ according to the complex of groups version of van Kampen's Theorem \cite[III.$\mathcal{C}$.3.11.(5), III.$\mathcal{C}$.3.12, p.552]{Bridson}.\footnote{One can also see this tree directly by considering the dual tree to the collection of hyperplanes of $X$ which project to $H$.  See \cite[Remark 1.1]{groves2018hyperbolic} for more details.}   The edge group of such a splitting is a hyperplane stabilizer for the $\Gamma_0$--action on $X$, which is relatively quasi-convex by \cite[Corollary 4.11]{einstein2021separation}, and infinite-index since the hyperplane is essential.  The action of $\Gamma_0$ on the Bass--Serre tree $T$ associated to this splitting has finite kernel, since any normal subgroup contained in an infinite-index relatively quasi-convex subgroup is finite. Let $F$ denote the kernel of the action of $\Gamma_0$ on $T$.

 By \cite{gromov1992harmonic}, the induced action of $\Gamma_0$ on $T$ factors through a surjective homomorphism $\varphi
 \co \Gamma_0 \rightarrow \Delta_0$, where $\Delta_0\leq \PSL_2(\R)$ is a cocompact lattice. The kernel of $\varphi$ is contained in $F$, hence finite, so $\Gamma_0$ is commensurable up to finite kernels with $\Delta_0$, which is itself virtually a hyperbolic surface group. Since any group commensurable up to finite kernels with a hyperbolic surface group is virtually a hyperbolic surface group, this means that $\Gamma_0$, and hence $\Gamma$, is virtually a hyperbolic surface group, as desired.
\end{proof}

\section{Relatively geometric actions: Lattices in $\PU(n,1)$}\label{sec:No Lattice Action}

Let $\Gamma$ be a non-uniform lattice in $\text{PU}(n,1)$. Then $\Gamma$ acts properly discontinuously on complex hyperbolic space $\Hy^n_\C$ and the quotient, which we henceforth denote by  $M=\leftQ{\Hy^n_\C}{\Gamma}$, is a non-compact orbifold of finite volume with finitely many cusps.  Each cusp corresponds to a conjugacy class of subgroups stabilizing a parabolic fixed point in $\partial_\infty\Hy^n_\C$.  Farb \cite{Farb-RH} proved that $\Gamma$ is hyperbolic relative to the collection of these cusp subgroups, which we denote by $\mathcal P$. In this section, we prove Theorem \ref{thm:NoRelativeAction}, namely that $(\Gamma,\cP)$ does not admit a relatively geometric action on a $\mathrm{CAT}(0)$ cube complex. 

Throughout the course of the proof, we pass freely to finite-index subgroups by invoking Lemma \ref{lem:Finite Index Action}. In order to streamline the exposition, we do not refer to Lemma \ref{lem:Finite Index Action} each time. First we reduce the Theorem \ref{thm:NoRelativeAction} to the case where $\Gamma$ is torsion-free.

\begin{lemma}\label{lem:Torsion Free SU(n,1)}
$\Gamma$ has a torsion-free subgroup of finite index.
\end{lemma}
\begin{proof}
We have a short exact sequence $$1\to \Z/(n+1)\Z\to \text{SU}(n,1)\to \text{PU}(n,1)\to 1.$$
 Restricting to $\Gamma$, we get a short exact sequence $$1\to \Z/(n+1)\Z\to \Lambda\to \Gamma\to 1,$$
 where $\Lambda$ is the pre-image of $\Gamma$ in $\text{SU}(n,1)$. Since $\Gamma$ is finitely generated and $\Z/(n+1)\Z$ is finite, $\Lambda$ is finitely generated. As $\SU(n,1)$ is linear, Selberg's lemma implies that $\Lambda$ has a finite-index torsion-free subgroup, say, $\Lambda_0$. Thus $\Lambda_0\cap \Z/(n+1)\Z=1$ and hence it is mapped isomorphically to finite-index subgroup $\Gamma_0\leq \Gamma$. 
\end{proof}

Following Lemma \ref{lem:Torsion Free SU(n,1)}, for the remainder of this section we assume that $\Gamma\leq \PU(n,1)$ is torsion-free.

\subsection{The structure of cusps}\label{sec:Cusp Structure}We now briefly review the geometric structure of cusps in $M$.  For more details see \cite{GoldmanBook}.
 Recall that up to scaling each horosphere in $\Hy^n_\C$ is isometric to $\mc{H}_{2n-1}(\R)$, the $(2n-1)$--dimensional real Heisenberg group, equipped with a left-invariant metric. The Heisenberg group is a central extension
 \begin{equation}\label{eqn:Heisenberg}1\rightarrow \R\rightarrow \mc{H}_{2n-1}(\R)\rightarrow \R^{2n-2}\rightarrow 1
\end{equation}
with extension $2$--cocycle equal to the standard symplectic form \[\displaystyle\omega=2\sum_{i=1}^{n-1}dx_i\wedge dy_i,\] where $(x_1,y_1,\ldots,x_{n-1},y_{n-1})$ are coordinates on $\R^{2n-2}.$ The Lie algebra $\mathfrak{h}_{2n-1}$ is $2$--step nilpotent with basis $\{X_1,~Y_1,\ldots,X_n,~Y_n,~Z\}$ where \[[X_i,Y_i]=Z\] and all other brackets vanish. Thus $Z$ generates the center of  $\mathfrak{h}_{2n-1}$ representing the kernel $\R$ in Equation \eqref{eqn:Heisenberg}, while the remaining coordinates project to 
 the generators of $\R^{2n-2}$. Choosing the identity matrix $I_{2n-1}$ as the inner product on $\mathfrak{h}_{2n-1}$, we see that the isometry group of $\mc{H}_{2n-1}(\R)$ is isomorphic to $\mc{H}_{2n-1}(\R)\rtimes U(n-1)$, where the $\mc{H}_{2n-1}(\R)$ factor is the action of $\mc{H}_{2n-1}(\R)$ on itself by left translation, and the unitary group $U(n-1)$ is the stabilizer of the identity. Indeed, any isometry which fixes $1\in \mc{H}_{2n-1}(\R)$ must also be a Lie algebra isomorphism; it therefore preserves the center $\langle Z\rangle$ and induces an isometry of $\R^{2n-2}\cong \langle X_1,Y_1,\cdots, X_{n-1},Y_{n-1}\rangle$ preserving $\omega$.  We conclude that such an isometry lies in $U(n-1)=O_{2n-2}(\R)\cap \Sp_{2n-2}(\R)$.
 
\begin{definition} Let $\pi\co  \mc{H}_{2n-1}(\R)\rtimes U(n-1)\rightarrow U(n-1)$ be the projection. For any $g\in \mc{H}_{2n-1}(\R)\rtimes U(n-1)$, we call $\pi(g)$ the \emph{rotational part} of $g$.
\end{definition}

Since the center of $\mc{H}_{2n-1}(\R)$ is invariant under any isometry we have a short exact sequence
\begin{equation}\label{eqn:Isometry group exact sequence}
    1\rightarrow \R=Z(\mc{H}_{2n-1}(\R))\rightarrow \mc{H}_{2n-1}(\R)\rtimes U(n-1)\rightarrow \R^{2n-2}\rtimes U(n-1)\rightarrow 1
\end{equation}

Since $\Gamma$ is torsion-free, each cusp subgroup $P\leq \Gamma$ is isomorphic to a discrete, torsion-free, cocompact subgroup of $\Isom(\mc{H}_{2n-1}(\R))$. In particular, $P_0=P\cap \mc{H}_{2n-1}(\R)$ is a discrete cocompact subgroup and $P\cap Z(\mc{H}_{2n-1}(\R))\cong\Z$. By Equation \ref{eqn:Isometry group exact sequence}, $P$ fits into a short exact sequence 
  \begin{equation}\label{eqn:cusp exact sequence}1\rightarrow \Z=P\cap Z(\mc{H}_{2n-1}(\R))\rightarrow P\rightarrow \Lambda \rightarrow 1\end{equation}
where $\Lambda$ is a discrete cocompact subgroup of $\R^{2n-2}\rtimes U(n-1)$. It follows that $\Lambda$ has a finite-index subgroup $\Lambda_0$ isomorphic to $\Z^{2n-2}$, which is the image of $P_0$.

On the level of quotient spaces, the sequence in Equation \eqref{eqn:cusp exact sequence} has the following translation.  The quotient space $\mathcal{O}=\leftQ{\R^{2n-2}}{\Lambda}$ is a Euclidean orbifold finitely covered by the $(2n-2)$--dimensional torus $T=\leftQ{\C^{n-1}}{\Lambda_0}$, and $\Sigma=\leftQ{\mc{H}_{2n-1}(\R)}{P}$ is  the total space of an $S^1$--bundle over $\mathcal{O}$, \emph{i.e.}, there is a fiber sequence 
\begin{equation}\label{eqn:Seifert Sequence}S^1\hookrightarrow \Sigma\rightarrow \mathcal{O}
\end{equation}
Since $\mathcal{O}$ need not be smooth, this is not generally a locally trivial fibration.  However, as $P$ is torsion-free, $\Sigma$ is smooth.  Passing to the torus cover, we obtain an actual fiber bundle \[S^1\hookrightarrow \widehat{\Sigma}\rightarrow T\]
The finite group $F=P/P_0$ acts on $\widehat{\Sigma}$ preserving the fibration, hence defines a finite group of isometries of $T$. Thus the stabilizer of a point in $T$ acts freely on the $S^1$ fiber. Since the action of $F$ on $\widehat{\Sigma}$ is free, it follows that point stabilizers in $T$ must be cyclic of finite order, and act by rotations on the fiber.   Since $F\leq U(n-1)$, any abelian subgroup is diagonalizable.  Thus, locally each point in $N$ has a neighborhood of the form $(S^1\times \D^{n-1})/(\Z/m\Z)$ where $\D\subset \C$ is the open unit disk, and $\Z/m\Z$ acts on $S^1$ by rotation by $2\pi/m$ and on the polydisk $\D^{n-1}$ by a diagonal unitary matrix $\displaystyle \Delta=\text{diag}\left(e^{\frac{2\pi k_1}{m}}, \ldots, e^{\frac{2\pi k_{n-1}}{m}}\right)$, where at least one $k_i$ is coprime to $m$. See Figure \ref{fig:Seifert} for a schematic.  Since $F$ acts by rotation on each fiber, $\Sigma$ is the boundary of a disk bundle over $\mathcal{O}$, which we denote by $E_\mathcal{O}$.

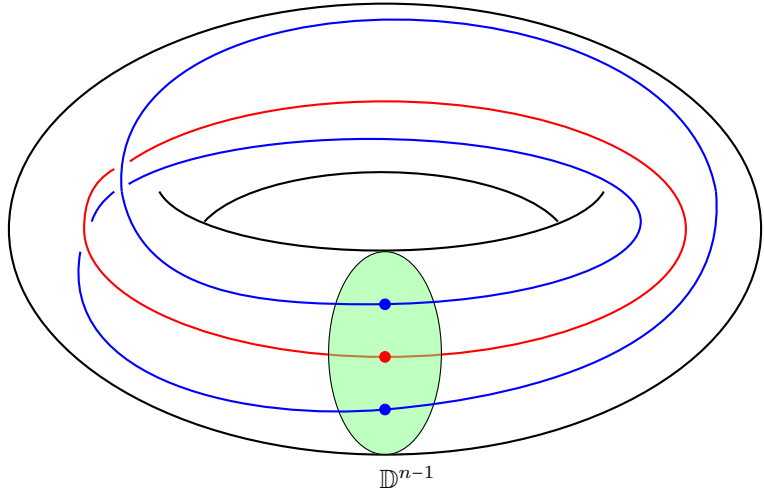
\begin{figure}[h]
    \centering
    \begin{tikzpicture}
        \draw[thick] (-5,0) arc(-180:180:5cm and 3cm);
        \draw[thick] (-3,.5) arc(-170:-10:3cm and .95cm);
        \draw[thick] (-2.4,.1) arc(160:20:2.5cm and 1cm);
        \draw[thick,red] (-4,0) arc(-180:148:4cm and 1.7cm);
        \draw[thick,red] (-4,0) to[out=90,in=210] (-3.6,.8);
        \filldraw[green!50,opacity=.5] (0,-.3) arc(90:450:.75cm and 1.35cm);
        \draw(0,-.3) arc(90:450:.75cm and 1.35cm);
      
        \draw[thick,blue](-.2,-1) arc(-90:153:3.6cm and 1.1cm);
        \draw[thick, blue](-3.6,.5) to[out=220,in=70] (-3.9,.1);
        \draw[thick, blue](-4.05,-.3) to[out=260,in=185] (0,-2.4);

        \filldraw[blue](0,-1) circle (2pt);
        \filldraw[red](0,-1.7) circle (2pt);
        \filldraw[blue](0,-2.4) circle (2pt);
        
       %  \draw[thick,blue](-.2,-1) arc(-90:170:2.55cm and 1.1cm);
       %   \draw[dotted,thick,blue] (-2.75,0.2) to[out=240,in=-200] (-2.2,-2.7);
       % \draw[thick,blue] (-2.1,-2.7) to[out=-10,in=-175] (0.25,-2.7);

       \draw[thick,blue] (0,-2.4) to[out=5,in=-85] (4.4,0.5);
       \draw[thick,blue] (4.4,0.5) to[out=100,in=95] (-3.5,0.5);
       \draw[thick,blue] (-3.5,0.5) to[out=-80,in=180](-.2,-1);

       \node at (.3,-3.3){$\mathbb{D}^{n-1}$};
     
       \end{tikzpicture}
    \caption{Local picture of the fibration in Equation \ref{eqn:Seifert Sequence} near a singular point of $\mathcal{O}$. A nonsingular fiber, shown in blue, winds $m=2$ times around the singular fiber, shown in red. }
    \label{fig:Seifert}
\end{figure}

Recall that the center of $\mc{H}_{2n-1}(\R)$ is quadratically distorted.  It follows that the center of $P$ is quadratically distorted as well.  By \cite[Theorem 1.5]{haglund2021isometries}, there is no proper action of $P$ on a $\mathrm{CAT}(0)$ cube complex.  Therefore, we have:

\begin{proposition}\label{prop:No Proper Action}
Let $\Gamma\leq \PU(n,1)$ be a non-uniform lattice, and suppose $\Gamma$ acts on a $\mathrm{CAT}(0)$ cube complex $X$. The action of each cusp subgroup of $\Gamma$ is not proper.  In particular, $\Gamma$ is not cubulated.
\end{proposition}

\subsection{The toroidal compactification of $M$}Another natural compactification of $M$ fills in the cusps with the Euclidean orbifolds described in Section \ref{sec:Cusp Structure}. Let $\mathcal{O}_i$ be the Euclidean orbifold quotient of $\Sigma_i$, with corresponding disk bundle $E_{i}$. Thus, we can identify $E_{i}\setminus \mathcal{O}_i$ with the cusp $\mathcal{C}_i$, then compactify $M$ by adding $\sqcup_i\mathcal{O}_i$ at infinity.  The result is a K\"ahler orbifold $\mathcal{T}(M)$ with boundary divisor $D=\sqcup_i \mathcal{O}_i$. The pair $(\mathcal{T}(M),D)$ is called the \emph{toroidal compactification of $M$}. See \cite{HummelSchroeder, Ashetal} for more details.

When the parabolic elements in $\Gamma$ have trivial rotational part, then  each $\mathcal{O}_i$ is a $(2n-2)$--dimensional torus, $\mathcal{T}(M)$ is a smooth K\"ahler manifold and $D$ is a smooth divisor in $\mathcal{T}(M)$. Moreover, Hummel--Schroeder show that $\mathcal{T}(M)$ admits a nonpositively curved Riemannian metric \cite{HummelSchroeder}.  In particular, $\mc{T}(M)$ is aspherical; if $\Delta=\pi_1(\mc{T}(M))$ then $\mc{T}(M)$ is a $K(\Delta,1)$. The following lemma ensures that we can always find a finite cover of $M$ whose toroidal compactification is smooth.

\begin{lemma}\label{lem: Smooth Toroidal Compactification} Let $\Gamma\leq \PU(n,1)$ be torsion-free and let $M=\leftQ{\Hy^n_\C}{\Gamma}$ be the quotient. There exists a finite cover $M'\rightarrow M$ such that the toroidal compactification of $M'$ is smooth.  
\end{lemma}
\begin{proof}
By the main theorem of \cite{hummel1998rank} (p. 2453), there exists a finite subset $F\subset \Gamma$ of parabolic isometries such that if a $N\unlhd\Gamma$ is a normal subgroup satisfying $F\cap N=\emptyset$, then any parabolic isometry in $N$ has no rotational part. Since $\Gamma$ is residually finite and $F$ is finite, we can find a finite-index normal subgroup $\Gamma'\unlhd \Gamma$ such that $\Gamma'\cap F=\emptyset$. 
Therefore the finite cover $M':=\leftQ{\Hy^n_\C}{\Gamma'}$ of $M$ admits a toroidal compactification which is smooth.
\end{proof}

For the rest of this section, we assume that $\mathcal{T}(M)$ is smooth.  Since $M_0\setminus \partial M_0\cong M\cong \mathcal{T}(M)\setminus D$, there is a natural map of pairs $f\co (M_0, \partial M_0)\rightarrow (\mathcal{T}(M),D) $ that is a diffeomorphism on the interior of $M_0$ and sends $\partial M=\sqcup \Sigma_i\rightarrow D=\sqcup_i\mathcal{O}_i$ via the fibering in Equation \ref{eqn:Seifert Sequence}.

\subsection{Proof of Theorem~\ref{thm:NoRelativeAction}}
We now have all the ingredients necessary to prove Theorem \ref{thm:NoRelativeAction}.
\begin{proof}[Proof of Theorem \ref{thm:NoRelativeAction}]
By Lemma \ref{lem:Torsion Free SU(n,1)} and Lemma \ref{lem: Smooth Toroidal Compactification}, we may assume that $\Gamma\leq \SU(n,1)$ is torsion-free, and that the toroidal compactification $\mathcal{T}(M)$ is smooth. In particular, $\Gamma$ and all of its peripheral subgroups are residually finite. 

Suppose $(\Gamma,\mathcal{P})$ admits a relatively geometric action on a $\mathrm{CAT}(0)$ cube complex $X$.  Given a finite-index subgroup $\Gamma_0\leq \Gamma$, let $\mathcal{P}_0$ be the induced peripheral structure on $\Gamma_0$, and let $\Delta_0$ be $\pi_1(\mathcal{T}(M_0))$, where $M_0=\leftQ{\Hy_\C^n}{\Gamma_0}$. Since the kernel of the quotient map $\Gamma_0\rightarrow \Delta_0$ is normally generated by subgroups in $\mathcal{P}_0$, we get an induced peripheral structure $(\Delta_0,\mathcal{A}_0)$, where $\mathcal{A}_0$ is the collection of images of elements of $\mathcal{P}_0$. Our strategy is to show that there exists a finite-index subgroup $\Gamma_0\leq \Gamma$ so that the pair $(\Delta_0,\mathcal{A}_0)$ is relatively hyperbolic and admits a relatively geometric action on a $\mathrm{CAT}(0)$ cube complex.  Since $\mathcal{T}(M_0)$ is smooth (since $\mathcal{T}(M)$ is), $\Delta$ is K\"ahler.  Thus, as $n\geq 2$, we will get a contradiction by Theorem \ref{thm:NoKahlerAction}.

Let $\mathcal{P}=\{P_1,\ldots, P_k\}$ be  the induced peripheral structure on $\Gamma$. 
Now let $Z(P_i)$ be the center of $P_i$. We apply Theorem \ref{thm:DehnFilling}(1) to a sufficiently long $\mathcal{Q}$--filling $\mathcal{Z}=\{Z_1,\ldots, Z_k\}$ where $Z_i\leq Z(P_i)$ is a finite-index subgroup. We then obtain a Dehn filling $\psi \co \Gamma \to \Delta = \Gamma / K$ determined by the $Z_i$ so that $Y=\leftQ{X}{K}$ is a $\mathrm{CAT}(0)$ cube complex.

Let $(\Delta,\mathcal{A})$ be the induced peripheral structure on $\Delta$.  By Theorem \ref{thm:DehnFilling}, we know that $(\Delta,\mathcal{A})$ is relatively hyperbolic.  Lemma~\ref{lem:quotientRG} implies that the action of $\Delta$ on $Y$ is relatively geometric.

Finally, we claim that there exists a finite-index subgroup $\Delta_0\leq \Delta$ that is torsion-free. Since the elements of $\mathcal{A}$ are virtually abelian, hence residually finite, we know that $\Delta$ is also residually finite by Corollary 1.7 of \cite{einstein2022relatively}. Since $\Gamma$ is torsion-free, by \cite[Theorem 4.1]{GM-elementary} so long as the filling $\Gamma \to \Delta$ is long enough (which we may assume without loss of generality), any element of finite order in $\Delta$  is conjugate into some element of $\mathcal{A}$. As there are finitely many elements of $\mathcal{A}$, each of which has only finitely many conjugacy classes of finite order elements, we can find a finite-index subgroup $\Delta_0\leq \Delta$ which avoids each of these conjugacy classes, hence is torsion-free. 
The induced peripheral structure $(\Delta_0,\mathcal{A}_0)$ is relatively hyperbolic and $\Delta_0\curvearrowright Y$ is relatively geometric by Lemma \ref{lem:Finite Index Action}. Let $\Gamma_0=\psi^{-1}(\Delta_0)$ and let $\mathcal{P}_0=\{P_{0,1},\ldots,P_{0,r}\}$ be the induced peripheral structure on $\Gamma_0$. Then $K\leq \Gamma_0$, and since $\Delta_0$ is torsion-free, this implies $K\cap P_{0,i}=Z(P_{0,i})$ for each $i$. As the $\mathcal{P_0}$ is the collection of cusp subgroups of $M_0=\leftQ{\Hy_\C^n}{\Gamma_0}$, we conclude that $\Delta_0=\pi_1(\mathcal{T}(M))$. Thus, $\Delta_0$ is K\"ahler and acts relatively geometrically on $Y$.  By Theorem \ref{thm:NoKahlerAction}, we conclude that $\Delta_0$ is virtually a hyperbolic surface group, which is impossible because  $\Delta_0$ contains a subgroup isomorphic to $\Z^{2n-2}$ and $n\geq 2$. This contradiction completes the proof.
\end{proof}

\begin{remark}
In \cite[Definition 1.9]{GrovesManningSpecializing}, Groves--Manning introduce the notion of a \emph{weakly relatively geometric action} on a $\mathrm{CAT}(0)$ cube complex. We can replace ``relatively geometric" with ``weakly relatively geometric" in Theorem~\ref{thm:NoRelativeAction} using similar arguments.  Indeed, after performing the toroidal filling of $(\Gamma,\cP)$ to land in the K\"ahler setting, we can perform a further peripherally finite filling to obtain a hyperbolic quotient, which is virtually special \cite[Theorem 4.5]{GrovesManningSpecializing}.

In this case the cube complex for the quotient is not $\leftQ{X}{K}$ where $K$ is the kernel of the filling homomorphism $\Gamma\rightarrow \overline{\Gamma}$.   Indeed, the action of $\overline{\Gamma}$ on $\leftQ{X}{K}$ in general has cell stabilizers that are virtually free. Nevertheless, Theorem D of \cite{groves2018hyperbolic} implies that in this case $\overline{\Gamma}$ is still cubulated hyperbolic.  The arguments from the remainder of the proof of Theorem~\ref{thm:NoRelativeAction} still apply.
\end{remark}

\bibliography{KahlerRelativelyGeometric} 
\bibliographystyle{plain} 

\end{document}